\documentclass[smallcondensed]{svjour3}     
\smartqed  
\usepackage{graphicx}
\usepackage{geometry,enumerate,datetime}                
\usepackage{graphicx}
\usepackage{amssymb,amsmath}
\usepackage{epstopdf}
\usepackage{array}
\usepackage{subfigure}
\usepackage{tensor,color}
\newcommand{\lap}{\bigtriangleup}
\newcommand{\C}{C_k}
\renewcommand{\S} {\mathcal{S}}
\newcommand{\bigO} {\mathcal{O}}

\newcommand{\e}{\varepsilon}

 \journalname{}
\begin{document}

\title{Integral equation methods for the Yukawa-Beltrami equation on
the sphere 
}



\author{M.C.~Kropinski    \and
        N.~Nigam          \and
        B.~Quaife}

\authorrunning{Kropinski, Nigam, Quaife} 

\institute{M.C.~Kropinski \at
              Department of Mathematics, Simon Fraser University \\
              \email{mkropins@math.sfu.ca}           
           \and
           N.~Nigam \at
              Department of Mathematics, Simon Fraser University \\
              \email{nigam@math.sfu.ca}           
          \and
          B.~Quaife \at 
              Institute for Computational and Engineering and Sciences, University of Texas \\
              Tel.: +1-512-232-3509 \\
              Fax: +1-512-471-8694 \\
              \email{quaife@ices.utexas.edu}
}


\maketitle

\begin{abstract}
An integral equation method for solving the Yukawa-Beltrami equation on
a multiply-connected sub-manifold of the unit sphere is presented. A
fundamental solution for the Yukawa-Beltrami operator is constructed.
This fundamental solution can be represented by conical functions.
Using a suitable representation formula, a Fredholm equation of the
second kind with a compact integral operator needs to be solved.  The
discretization of this integral equation leads to a linear system whose
condition number is bounded independent of the size of the system.
Several numerical examples exploring the properties of this integral
equation are presented. 
\end{abstract}

\section{Introduction}

Applications of partial differential equations (PDEs) on surfaces and
manifolds include image processing, biology, oceanography, and fluid
dynamics~\cite{Chaplain,Myers,Witkin}.  Since solutions of these PDEs
depend both on local {\it and} global properties of a given
differential operator on the manifold, standard numerical
discretization methods developed for PDEs in the plane or in
$\mathbb{R}^3$ need to be modified. Recent work in this direction
includes the closest point method~\cite{ruuth}, surface
parametrization~\cite{floater}, embedding functions,~\cite{Bertalmio},
and projections onto an approximation of the manifolds by tesselations
of simpler, non-curved domains (such as triangles)~\cite{lindblom}. 

It is well-known that for elliptic PDEs in $\mathbb{R}^2$ or
$\mathbb{R}^3$, numerical methods based on integral equation
formulations can offer significant advantages: efficiency is achieved
through dimension reduction, and superior stability of such methods
allows highly accurate solutions to be computed.  In addition, the
development of efficient numerical techniques such as the fast
multipole method or fast direct solvers have made integral equation
approaches significantly faster than many other currently available
schemes.  Relatively little work has been done, however, on using
integral equation methods for numerically investigating elliptic PDEs
on subsurfaces of manifolds.  Some prior work in this direction for the
Laplace-Beltrami operator on the surface $\cal{S}$ of the unit sphere
was presented in~\cite{gemmrich,kro:nig2013}.

In this paper we present a reformulation of a boundary value problem
for the Yukawa-Beltrami equation on the surface of the unit sphere
$\S$  in terms of boundary integral equations. Concretely, let $\Omega$
denote a sub-manifold of $\S$, and let $\Gamma$ denote its boundary. By
this we mean that $\Gamma$ is a closed curve  on $\S$ which divides
$\S$ into two (not necessarily connected) parts $\Omega$ and
$\Omega^{c}$. In particular, $\Gamma = \partial\Omega$ is the boundary
curve of $\Omega$. We wish to solve the Yukawa boundary value problem:
\begin{subequations}
\label{bigmodel}
  \begin{align}
    -\lap_\S u(x)+k^2  u(x)&= f({ x}), && \mbox{for} \: x \in \Omega, \\
    u(x)&=g(x), &&\mbox{for} \; x \in \Gamma. 
  \end{align}
\end{subequations}
Here $\lap_{\S}$ is the Laplace-Beltrami operator on $\S$ and $k>0 \in
\mathbb{R}$ is constant. This  problem arises in the context of solving
the isotropic heat equation via Rothe's method. By first applying a
semi-implicit time discretization to the heat equation,  time-stepping
then involves repeatedly solve an elliptic PDE of the
form~\eqref{bigmodel}, where $k^2 = O(1/\Delta t)$.  This approach is
discussed in~\cite{rothe:heat} for the heat equation in the plane. 

We recall that the exterior boundary value problem for the Yukawa
operator $(-\lap + k^2)$  in $\mathbb{R}^2$ is well-posed if we seek
$H^1$ (weak) solutions, even without the specification of a radiation
condition (see, e.g.,~\cite{gatica}). This is in contrast to the
exterior problem for the Laplacian.  In~\cite{gemmrich}, it was observed
that the single-layer operator for the Laplace-Beltrami did not satisfy
the associated boundary value problem on $\S$ unless a further
constraint was satisfied. Experience with the Yukawa operator in
$\mathbb{R}^2$~\cite{kro:qua2011,qua2011} suggests that any issues
concerning unique solvability which arise for the Laplace-Beltrami
operator when we move to a compact manifold will be ameliorated for the
Yukawa-Beltrami operator.  We shall see this is indeed the case: it is
possible, provided $k>\frac{1}{2}$, for the single-layer operator to
exactly satisfy the boundary value problem. (In case we solve the
Yukawa-Beltrami problem on a sphere of radius $R$, we require $kR>1/2$.)

To simplify the exposition and analysis, we shall concentrate on
locating the homogeneous solution of~\eqref{bigmodel}.  In other words:
{\it Find a smooth $u$ such that for given smooth Dirichlet data $g$}
\begin{subequations}
  \label{DBVP}
  \begin{align}
    -\lap_{\S} u(x) +k^2u(x)\, &= \, 0, &&\mbox{for} \; 
      x \in { \Omega},\\
    u(x) &= g(x), &&\mbox{for} \; x \in \Gamma.
  \end{align}
\end{subequations}
We note that we could equivalently have chosen to study the Neumann or
Robin problem for the system.  We wish to solve~\eqref{DBVP} by
reformulating this boundary value problem as an integral equation. As is
expected, the process of reformulation is not unique; we shall be
employing a layer ansatz based on a parametrix for the Yukawa-Beltrami
operator, and solving an integral equation for an unknown density.  The
choice of parametrix is not unique, and we derive a particularly
convenient parametrix involving conical functions. By proceeding with a
double-layer ansatz based on this parametrix, a well-conditioned
Fredholm equation of the second kind results. Several numerical examples
are presented which illustrate the analytic properties of this integral
equation.

\subsection{Some preliminaries}
 We favour an intrinsic definition, and where possible identify $x\in \S$ by two independent variables (the spherical angles), $x=x(\phi,\theta)$. We can also describe this point in terms of a Euclidean coordinate system in $\mathbb{R}^3$ whose origin is at the center of mass of the sphere: 
\begin{align*}
 x= x(\varphi,\theta) \, \equiv \, \left( 
  \begin{array}{c}
    \cos \varphi \sin \theta \\
    \sin \varphi \sin \theta \\
    \cos \theta
  \end{array} 
  \right) \in{\S}, \quad \varphi \in [0,2\pi), 
    \quad \theta \in [0,\pi].
\end{align*}
We also recall that on $\S$, the Laplace-Beltrami operator  $\lap_\S$
is defined as
\begin{align*}
  \lap_S u(x) \, = \, \left[
  \frac{1}{\sin^2 \theta} \frac{\partial^2}{\partial \varphi^2} +
  \frac{1}{\sin \theta} \frac{\partial}{\partial \theta}
  \left(\sin \theta \frac{\partial}{ \partial \theta}\right)
  \right] u(x(\varphi,\theta)).
\end{align*}

If two points $x,y$ lie on the unit sphere, with
spherical coordinates $x=x(\phi,\theta), y=(\alpha,\beta)$ then we describe the solid angle between them (a measure of their distance in the metric on $\S$) by
\begin{align*}
  <x,y> =\cos(\phi-\alpha)\sin(\theta)\sin(\beta)+
    \cos(\theta)\cos(\beta).
\end{align*} 
In particular if $y$ is the North Pole then  $\beta=0$ and
$<x,y>=\cos(\theta)$. (If we denoted $x,y$  by their Cartesian
coordinates $(x_1,x_2,x_2),(y_1,y_2,y_3)$, then $<x,y> = \sum_{i=1}^3
x_iy_i$.) The distance between $x$ and $y$ in the 3-dimensional
Euclidean metric $\|x-y\|$ is given by 
\begin{align*}
  \|x-y\|^2 := 2-2<x,y>, \qquad \Rightarrow \qquad 
    <x,y>= 1-\frac{\|x-y\|^2}{2}.
\end{align*}

We remind the reader of some useful vectorial identities on the sphere.
Let $ \vec{e}_\theta, \vec{e}_{\varphi} $ be the usual unit vectors in
spherical coordinates.  Recall that we can define the surface gradient
of a scalar $f$ on $\S$ as
\begin{align*}
  \nabla_{\S} f(x) \,  = \, \frac{1}{\sin \theta}
  \frac{\partial f}{\partial \varphi} \,\vec{e}_{\varphi} + 
  \frac{\partial f}{\partial \theta}\, \vec{e}_{\theta}.
\end{align*}
In the same way, the surface divergence for a vector-valued
function $\vec{V}$ on the sphere can be written as
\begin{align*}
  \mbox{div}_{\S} \vec{V}(x) \, = \, 
  \frac{1}{\sin \theta} \left(
  \frac{\partial}{\partial \varphi}
  V_{\varphi}(\varphi,\theta) +
  \frac{\partial}{\partial \theta} ((\sin\theta) \, 
  V_{\theta}(\varphi,\theta) ) \right).
\end{align*}
We easily see the identity $\lap_{\S}u(x) \, = \,
\mbox{div}_{\S}\nabla_{\S} u(x)$.  The vectorial surface rotation for a
scalar field $f$ on the sphere is 
\begin{align*}
  \underline{\mbox{curl}}_{\S} f(x) \, = -\, 
  \frac{\partial f}{\partial \theta} \,\vec{e}_{\varphi} + 
  \frac{1}{\sin\theta}\frac{\partial f}{\partial \varphi}\, 
  \vec{e}_{\theta},
\end{align*}
and the (scalar) surface rotation of a vector field $\vec{V}$ is
\begin{align*}
  \mbox{curl}_{\S} \vec{V}(x) \, = \,
  \frac{1}{\sin \theta} \left(
  - \frac{\partial}{\partial \varphi} V_{\theta}(\varphi,\theta) +
  \frac{\partial}{\partial \theta} ((\sin \theta) \,   
  V_{\varphi}(\varphi,\theta)) \right).
\end{align*}
We then obtain another vectorial identity for the Laplace-Beltrami
operator:
\begin{align*}
  \lap_{\S} u(x) \, = \, - \mbox{curl}_{\S} 
  \underline{\mbox{curl}}_{\S} u(x) \quad \mbox{for} \; x \in {\S}.
\end{align*}

Stoke's theorem for the  smooth positively oriented curve $\Gamma$ and
the enclosed region $\Omega$ may be written  as
\begin{align*}
  \int\limits_{\Omega} \mbox{curl}_{\S} \vec{V}(x) d\sigma_x \, = \,
  \int\limits_\Gamma \vec{V}(x) \cdot \vec{t}(x) \, ds_x .
\end{align*}
Here, $\vec{t}$ is the unit tangent vector to $\Gamma$, $d\sigma$ is an
area element, and $ds$ is an arclength element. We note that a similar
identity holds for the region $\Omega^{c}$, with care taken with the
orientation of the tangent.  Now, setting $\vec{V} = v(x) \vec{W}(x)$
and applying the product rule we have
\begin{align*}
  \int\limits_{\Omega} \underline{\mbox{curl}}_{\S} v(x) \cdot
  \vec{W}(x) \, d\sigma_x \, = \, 
  -\int\limits_{\Gamma} v(x) [\vec{W}(x) \cdot \vec{t}(x)] ds_x +
  \int\limits_{\Omega} v(x) \mbox{curl}_{\S} \vec{W}(x) d\sigma_x .
\end{align*}
With $\vec{W}(x) = \underline{\mbox{curl}}_{\S} u(x)$ we finally
obtain Green's first formula for the Laplace-Beltrami operator
$\lap_{\S}$,
\begin{align}
  \label{Green1}
  -\int\limits_{\Omega} \underline{\mbox{curl}}_{\S} v(x) \cdot
  \underline{\mbox{curl}}_{\S} u(x) \, d\sigma_x 
  =& \int\limits_{\Gamma} v(x) [\underline{\mbox{curl}}_S u(x) \cdot 
  \underline{t}(x)] ds_x +\int\limits_{\Omega} v(x) \lap_{\S}u(x) d\sigma_x.
\end{align}

We obtain Green's second formula by interchanging the roles of $u$ and
$v$ in~\eqref{Green1}, subtracting the two identities, and using the
symmetry of the left hand side
\begin{align}
  &\int\limits_{\Omega} v(x)(-\lap_{\S}u(x) + k^2 u(x))-
  u(x)(-\lap_{\S}v(x) +k^2v(x))\, d\sigma_x \nonumber \\
  =&\int\limits_\Gamma [v(x)\,\underline{\mbox{curl}}_{\S} 
  u(x) - u(x)\, \underline{\mbox{curl}}_{\S}v(x)] \cdot 
  \vec{t}(x) ds_x.
  \label{Green2}
\end{align}

We shall make extensive use of these identities.

\section{A fundamental solution and representation formula}
We seek a fundamental solution for the Yukawa-Beltrami operator
$(-\lap_\S + k^2)$.  Examining first the situation for the Yukawa
operator in the Euclidean plane, the fundamental solution of the Yukawa
operator, $(-\lap + k^2)$ in $\mathbb{R}^2$ is given by
$\frac{k^2}{2\pi} K_{0}(kr)$, where $r=\|\mathbf{x} - \mathbf{x}_{0}\|$
is the distance between the source and evaluation point in the
two-dimensional Euclidean metric. Here $K_{0}$ is the modified Bessel
function of order 0 which is analytic for non-zero argument, and has a
logarithmic singularity when the source and evaluation points coincide,
that is, when $r=0$.

On the surface of the sphere $\S$, we expect the fundamental solution
$G_k(x,x_{0})$ for the Yukawa-Beltrami operator to possess a
logarithmic singularity when $x$ approaches $x_{0}$. Exactly as was
done  with the Laplace-Beltrami operator in~\cite{gemmrich}, we could
define a parametrix for this operator by using the {\it distance
measured in the Euclidean norm in $\mathbb{R}^3$}. This would suggest
using $ \frac{k^2}{2\pi}K_0(k\|x-x_{0}\|)$.  Such a choice of
parametrix would be directly related to the fundamental solution of the
two-dimensional operator; the amount by which it fails to yield a dirac
measure is directly attributable to the difference between the
spherical and flat metrics.  In particular, with $r:=\|x-x_{0}\|$, we
have
\begin{align*}
  (-\lap_{\S} + k^{2})\frac{k^2}{2\pi}K_0(kr) = \delta(r) 
  +\frac{k^{4}r^{2}}{8\pi}K_{0}(kr) - \frac{k^{3}r}{4\pi}K_{1}(kr).
\end{align*}

While this is a perfectly reasonable parametrix, it is inconvenient
from the perspective of rewriting the Yukawa-Beltrami boundary value
problem in terms of a boundary integral equation. The term
$\frac{k^{4}r^{2}}{8\pi}K_{0}(kr) - \frac{k^{3}r}{4\pi}K_{1}(kr)$ will
result in {\it volumetric} constraints  appearing in an integral
equation representation. Such a term encapsulates the fact that we are
on a compact manifold, no longer on $\mathbb{R}^2$. Recall that for the
Laplace-Beltrami case, when $k=0$, this term
$\frac{k^{4}r^{2}}{8\pi}K_{0}(kr) - \frac{k^{3}r}{4\pi}K_{1}(kr)$
reduces to a constant; it is then possible to use a double-layer ansatz
in which no volumetric terms or constraints appear~\cite{kro:nig2013}.  

To avoid such complications, we shall instead derive a more convenient
parametrix. To the best of our knowledge, the use of this parametrix
and the associated boundary integral equations is novel. 

Without loss of generality we first set $x_{0}$ to be the point at the
north pole.  Let $x$ be some other point on the sphere. The parametrix
$G_k(x,x_0)$ we seek depends on
$r_0(\theta)=\|x-x_0\|=\sqrt{2-2\cos(\theta)}$. Since there is no
angular dependence in $\phi$, the Yukawa-Beltrami operator reduces to
the second order ODE operator
\begin{align*}
  \mathcal{D}_k(u) := \frac{1}{\sin(\theta)}
  \frac{d}{d\theta}\left(\sin(\theta)
  \frac{d}{d\theta}u(r_0(\theta))\right)-k^2 u(r_{0}(\theta)).
\end{align*}
A simple change of variables allows us to rewrite $\mathcal{D}_k (u)=0$
as 
\begin{align}
  (1-z^2)\,w'' -2zw' + \left[\nu(\nu+1)\right]\,w = 0,
  \label{LegendrePequation}
\end{align}
where $\nu = \nu(k)$ satisfies $\nu(\nu + 1) = -k^2$. We note here that
the Helmholtz-Beltrami operator on the sphere would lead to to the same
equation, but with $\nu(k)$ satisfying $\nu(\nu+1)=k^{2}$. In what
follows we shall suppress the dependence of $\nu$ on $k$ when there is
no risk of confusion.

Equation~\eqref{LegendrePequation} is the well-known {\it Legendre's
equation}, which is well-defined for arbitrary real or complex $\nu$.
Therefore, we can locate two linearly independent solutions
of~\eqref{LegendrePequation}, the so-called {\it first and second kind
Legendre functions of degree $\nu(k)$}, denoted $ P_{\nu(k)}(z)$ and
$Q_{\nu(k)}(z)$, respectively. Of these, only the  LegendreP function
$P_{\nu(k)}(z)$ remains finite as $z$ approaches $1$ (to a limiting
value of 1 as we will soon see),~\cite{lebedev}.  The function
$P_{\nu(k)}(z)$ is well-defined for $|1-z|<2$.  This means that
\begin{align*}
  u(r_0(\theta)) =  P_{\nu(k)}(-\cos(\theta))
\end{align*}
is well-defined whenever $|1+\cos(\theta)|<2$, which holds for all
$\theta \in(0,\pi]$ (see Figure~\ref{f:legendreP}).  Moreover, since we
want $u(r_0(\pi))$ to be finite, we do not use the other possible
solution of~\eqref{LegendrePequation}, namely $Q_{\nu(k)}$.
\begin{figure}[htps]
  \centering
  \includegraphics[width=0.4\textwidth]{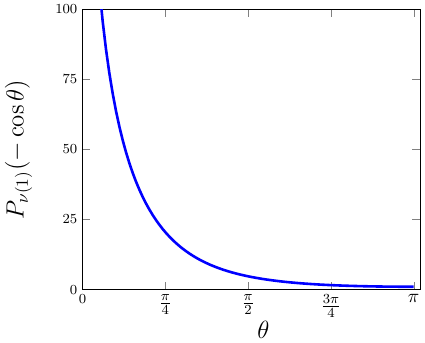}
  \caption{\label{f:legendreP} The LegendreP function
  $P_{\nu(k)}(-\cos\theta)$, for $\theta \in (0,\pi]$ and $k=1$}
\end{figure}  
We also recall that $P_{\nu}$ is well-defined for all $\nu \in
\mathbb{C}$, whereas $Q_\nu$ is well-defined for $\nu \not=
0,-1,-2,...$. 

On $\S$, from~\eqref{LegendrePequation} we obtain
\begin{align*}  
  \nu(k):=\frac{-1+\sqrt{1-4k^2}}{2}.
\end{align*}
If $0<k^2<\frac{1}{4}$ then $\nu \in (-1,0)$. If $k^2\geq\frac{1}{4}$ then 
\begin{align*}  
  \nu(k):=\frac{-1+i\sqrt{4k^2-1}}{2}.
\end{align*}

We observe that since $k^{2}>0$, $\nu(k) \notin \mathbb{Z}$. Moreover,
if $k>\frac{1}{2}$ then $Im(\nu)\not=0$. This is the regime in which we
shall work, and in what follows we shall make this assumption on $k$.
Recall that in $k^{2} = bigO(1/\Delta t)$ in the our applications of
interest.

A further substitution allows us to write~\eqref{LegendrePequation} as
a hypergeometric equation, (see, e.g., Section 7.3,~\cite{lebedev}).
This in turn allows us to write the solution of $\mathcal{D}_k(u)=0$ as 
\begin{align*} 
  u(r_0(\theta)) &= P_{\nu(k)}(-\cos(\theta)) = 
    \tensor[_2]{F}{_1}\left(-\nu(k), \nu(k)+  1; 1; 
      \frac{4-r_0^2(\theta)}{2}\right) \\
    &=\tensor[_2]{F}{_1}\left(-\nu(k), \nu(k)+  1; 1;
    \frac{1+\cos(\theta)}{2}\right).
\end{align*} 
Here, $\tensor[_2]{F}{_1}(a,b;c;z)$ is a hypergeometric function. Since
the argument $\frac{1+\cos(\theta)}{2}$ lies between $-1$ and $1$, and
$a+b = 1$, the representation in terms of the hypergeometric function is
also known as the {\it Ferrers' function of the first
kind},~\cite{fatAbramowitz}. 

\begin{figure}
  \centering
  \includegraphics[width=0.6\textwidth]{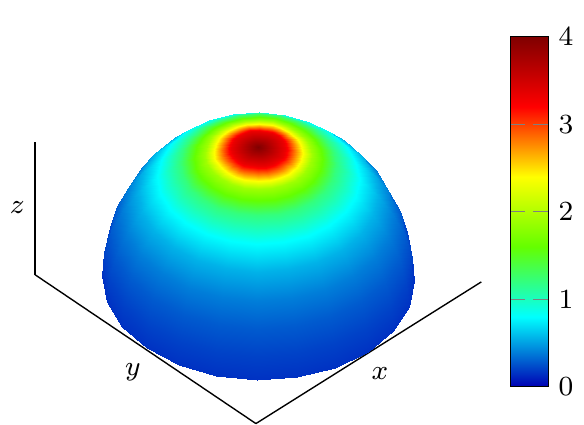}
  \caption{\label{f:greensFun} $G_1(x,x_0)$ when $x_0=(0,0,1)^T$. The
  singularity at the north pole is logarithmic. Only the top hemisphere
  is shown here for purpose of illustration}
\end{figure}
A more (computationally) convenient representation is in terms of the
{\it conical or Mehler functions} (Section 8.5 in~\cite{lebedev}). The
conical function $P_{-\frac{1}{2}+ i \tau}(z)$ of order $\tau\in
\mathbb{R}$ solves the equation
\begin{align*}
  (1-z^2)\,w'' -2zw' - \left( \tau^2+\frac{1}{4}\right) \,w = 0.
\end{align*}
In our specific case, we would pick $\tau = \tau(k)$ such that
$\tau^2+\frac{1}{4}=k^2 = -\nu(\nu+1).$  It follows that 
\begin{align*}
  \tau:=\frac{\sqrt{4k^2-1}}{2} \Rightarrow  \nu= -\frac{1}{2}+i\tau, 
\end{align*}
and the solution of $\mathcal{D}_k(u)=0$ is $
  P_{\nu} (-\cos(\theta)) = P_{-\frac{1}{2} + i \tau} 
    (-\cos(\theta)).$
The calculations above lead us then to the following definition:

\begin{definition}
\label{fundamentaldef} 
The fundamental solution for the Yukawa-Beltrami operator $-\lap_\S +
k^2$ for points $x,x_0$ on the surface of the unit sphere is 
\begin{subequations}
  \begin{align}
    G_k(x,x_{0})&:=\C P_{\nu}\left(-<x,x_{0}>\right) \nonumber \\
    &:=\C P_{\nu}\left(\frac{\|x-x_{0}\|^2}{2}-1\right) \label{LegP} \\
    &= \C\, \tensor[_2]{F}{_1} \left(-\nu, \nu+1; 1; 
      \frac{1+<x,x_{0}>}{2}\right)\label{hyperG} \\
    &=\C P_{-\frac{1}{2} + i \tau}
    \left(<x,x_{0}>\right) \label{conicalP},
  \end{align} 
\end{subequations}
where 
$\nu:=\frac{-1+i\sqrt{4k^2-1}}{2}$, $k>1/2$, and 
\begin{align}
  \label{constant-definition}    
  \C=-\frac{1}{4\sin(\nu\pi)} 
    =\frac{1}{4\cosh(\frac{\pi}{2}\sqrt{4k^2-1})}.
\end{align}
\end{definition} 
(See Figure~\ref{f:greensFun} for a visualization of $G_1(x,x_0)$.) We
mention here that the specific choice of $\C$ is motivated by the
calculations performed while deriving a representation formula in
Section~\ref{s:representation}.  We will see that $\C$ is well-defined
with our assumption that $k>1/2$.

\subsection{Properties of the fundamental solution of the
Yukawa-Beltrami operator}
Before we embark on the definition and analysis of boundary integral
operators, we examine some properties of the fundamental solution $G_k$
defined in Definition~\ref{fundamentaldef}. We shall use either the
representation~\eqref{LegP},~\eqref{hyperG}, or~\eqref{conicalP} as
convenient.  We first observe that the fundamental solution is symmetric
in its arguments, $G_k(x,x_{0}) = G_k(x_{0},x)$. Next, by noting the
expansion for the conical functions as~\cite{lebedev}
\begin{align*}
  P_{-\frac{1}{2}+i \tau} (\cos(t)) &= 1+
    \frac{4\tau^2+1}{2^2}\sin^2(t/2) + 
      \frac{(4\tau^2+1)(4\tau^2+3^2)}{2^24^2}\sin^4(t/2) + 
    \cdots,
\end{align*}
for $0\leq t\leq \pi$, we see that $G_k(x,x_0)$ does not change sign for
  all $x,x_{0} \in \S, x \neq x_{0}$, provided $k$ is fixed.

Next, we examine the asymptotic behaviour as $x\rightarrow x_0$. In this
case, it is convenient to work with~\eqref{LegP}. Setting $\mu=0$ in
14.8.1 of~\cite{fatAbramowitz} (or using 7.5.5 in~\cite{lebedev}), we
have the asymptotic behaviour
\begin{align} 
  \label{Pasymptotics}
  \lim_{t\rightarrow 1^- } P_\nu(t)&= \frac{1}{\Gamma(1)}=1.
\end{align}
Next, following~\cite{fatAbramowitz}, as $t\rightarrow 1^{-}$,
\begin{align*}
  Q_{\nu}(t)=\frac{1}{2}\log\left(\frac{2}{1-t}\right)-\gamma
  -\psi(\nu+1)+\bigO(1-t),\: \nu \neq -1,-2,...
\end{align*} 
where $\psi(x)$ is the digamma function, $\gamma=0.5772\ldots$ is
Euler's constant, and $Q_{\nu}$ is the other linearly independent
solution of Legendre's equation~\eqref{LegendrePequation}.  For
subsequent use, we denote 
\begin{align*}
  R(\nu):= \frac{1}{2}\log(2) -\gamma -\psi(\nu+1), 
\end{align*}
and therefore for $t$ close to $1$,
\begin{equation}
  \label{Qasymptotics}
  Q_{\nu}(t)=-\frac{1}{2}\log(1-t) + 
  R(\nu) + \bigO(1-t),\: \nu \neq -1,-2,... 
\end{equation}

Now, as $x\rightarrow x_{0}$ on the surface of the sphere,
using~\eqref{LegP} we have 
\begin{align*}
  \lim_{x\rightarrow x_{0}}G_k(x,x_{0}) =
  \C\lim_{x\rightarrow x_{0}} P_{\nu(k)}\left(-<x,x_{0}>\right) =
  \C\lim_{t\rightarrow 1^-}P_{\nu(k)}\left(-t\right).
\end{align*}
We cannot immediately use~\eqref{Pasymptotics}. Instead, we must use the
connection formula 14.9.10 in~\cite{fatAbramowitz} (again setting
$\mu=0$):
\begin{align}
  \frac{2}{\pi}\sin(\nu \pi)Q_\nu(t)= \cos(\nu\pi)P_\nu(t) - P_\nu(-t).
  \label{connect}
\end{align}
Combining~\eqref{connect} with~\eqref{Pasymptotics}
and~\eqref{Qasymptotics}, we obtain
\begin{align}
  \lim_{x\rightarrow x_{0}}G_k(x,x_{0}) &= \lim_{x\rightarrow x_{0}}
      \C P_{\nu}(-<x,x_{0}>) \nonumber \\
  &= \lim_{x\rightarrow x_{0} }\C\left(\cos(\nu\pi)
  P_{\nu}(<x,x_{0}>) - 
      \frac{2}{\pi}\sin(\nu \pi)Q_{\nu}(<x,x_0>)\right) \nonumber \\
  &= \C\cos(\nu\pi) \nonumber \\
  & \quad - \frac{2\C}{\pi}\sin(\nu\pi)\lim_{x\rightarrow x_0} 
      \left(-\frac{1}{2}\log(1-<x,x_0>) + R(\nu) +
      \bigO(1-<x,x_0>)\right) \nonumber \\
  &=\frac{\C\sin(\nu\pi)}{\pi}\lim_{x\rightarrow x_0}
      \log(1-<x,x_0>) + \C\left[\cos(\nu\pi)-
      2\frac{\sin(\nu\pi)}{\pi}R(\nu)\right].
  \label{e:logSingularity}
\end{align}
Therefore, as $x\rightarrow x_{0}$ $G_k(x,x_{0})$ possesses a
logarithmic singularity (recall that $\nu$ is not an integer).  At this
point, we could predict a value for $\C$ so that the strength the
singularity is consistent with that for the
Laplace-Beltrami~\cite{gemmrich}.  However, we compute $\C$ rigorously
below.

It is also useful to document the relation, obtained using recurrence
relations for the LegendreP function (Section 14.10.4
in~\cite{fatAbramowitz}):
\begin{align*} 
  (1-z^2) P'_\nu(z) = (-\nu-1) P_{\nu+1}(z) + (\nu+1)z P_\nu(z).
\end{align*}

Suppose, again without loss of generality, that $x_0$ is the North Pole.
Let $x=x(\theta,\phi)$. Then, $G_k(x,x_0) = \C P_\nu(-\cos(\theta))$,
and
\begin{align}
  \label{difflegp} 
  \frac{d}{d\theta} P_\nu(-\cos(\theta)) =  
    (\nu+1)(\sin\theta) \left( \frac{ P_{\nu+1}(-z) +z
        P_{\nu}(-z)}{z^2-1} 
   \right)\vert_{z=\cos(\theta)}.
\end{align}
Now, using~\eqref{difflegp}, 
\begin{align*}
  \underline{\mbox{curl}}_{\S} G_k(x,x_0)&=-
  \C\frac{\partial}{\partial \theta} \left[P_{\nu(k)}(-\cos(\theta))
  \right] \vec{e}_\varphi \\
  &=-\C(\nu+1)\frac{\sin(\theta)}{\cos^2\theta -1} 
    \left[ P_{\nu+1}(-z) + 
   z 
    P_\nu(-z)\right]\vert_{z=\cos(\theta)} \vec{e}_\varphi. \\
\end{align*}
Since the line element on the surface of the sphere is given by
\begin{align*}
  \vec{t}(x(\varphi,\theta))\cdot ds_{x(\varphi,\theta)} = 
    d\theta\,\vec{e}_{\theta} \,+
    \, \sin\theta d\varphi \,\vec{e}_{\varphi}, 
\end{align*} 
we obtain that
\begin{align}
  \label{curlT}
  \underline{\mbox{curl}}_{\S} G_k(x,x_0)\cdot \vec{t} &= 
  \C (\nu+1)  \left[ P_{\nu+1}(-z) + 
  z P_\nu(-z)\right]\vert_{z=\cos(\theta)} d\phi =: F(\theta)d\phi.
\end{align}
The term $F(\theta)= \C (\nu+1)  \left[ P_{\nu+1}(-z) + z
P_\nu(-z)\right]\vert_{z=\cos(\theta)}$ will be important in subsequent
sections, and is related to the kernel of the second layer operator.
(We shall define this operator carefully later.) For now, however, we
are interested in the limit of $F(\theta)$ as the point $x\rightarrow
x_0$, that is, as $\theta \rightarrow 0^+$.  

\begin{lemma} 
Let $F(\theta)$ be as defined in~\eqref{curlT}. Then
$\lim_{\theta\rightarrow 0^+} F(\theta)= \frac{2}{\pi}C_k \sin(\nu
\pi)$.
\end{lemma}
\begin{proof} 
Using~\eqref{connect}, the term in the square bracket of
\begin{align*}
F(\theta) =  \C (\nu+1)  \left[ P_{\nu+1}(-z) + 
  z P_\nu(-z)\right]\vert_{z=\cos(\theta)},
\end{align*}
can be simplified.  Denoting $\cos(\theta)=z$, 
\begin{align}
  P_{\nu+1}(-z) + z P_\nu(-z)&=-\frac{2}{\pi}\sin\left((\nu+1) \pi\right)Q_{\nu+1}(z)+
    \cos((\nu+1)\pi)P_{\nu+1}(z) \nonumber \\
  & \qquad +z\left[-\frac{2}{\pi}\sin(\nu\pi)Q_\nu\left(z\right)+
    \cos(\nu\pi)P_\nu\left(z\right)  \right] \nonumber \\
  &=-\frac{2}{\pi}\left[z\sin(\nu \pi)Q_\nu\left(z\right) + 
    \sin\left((\nu+1) \pi\right)Q_{\nu+1}\left(z\right)\right]
    \nonumber \\
  & \qquad +\left[ \cos((\nu+1)\pi)P_{\nu+1}\left(z\right) + 
    z\cos(\nu\pi)P_\nu\left(z\right) \right] \nonumber \\
  &=-\frac{2}{\pi}\sin(\nu \pi)\left[ 
    zQ_\nu\left(z\right) - Q_{\nu+1}\left(z\right)\right] 
  + \cos(\nu\pi)\left[
    -P_{\nu+1}\left(z\right)+zP_\nu\left(z\right)\right].
  \label{bigmess1}
\end{align}
 
As $x\rightarrow x_{0}$, i.e., as $\theta \rightarrow 0^{+}$, clearly
$z=\cos(\theta)\rightarrow 1^{-}$. Recall that in this limit, $P_\nu(z)
\rightarrow 1.$ Therefore, the terms involving the LegendreP functions
of the first kind in~\eqref{bigmess1} behave as 
\begin{align} 
  \label{smallermess1}
  \lim_{z \rightarrow 1^- }\cos(\nu\pi)\left[ -P_{\nu+1}(z)+ z P_\nu(z)
  \right]  = 0.
\end{align}
We can use~\eqref{Pasymptotics} to compute the limit of the term
involving the Legendre functions of the second kind in~\eqref{bigmess1}:
\begin{align}
  \lim_{z \rightarrow 1^{-}} 
    -\frac{2}{\pi}\sin(\nu\pi)\left[zQ_\nu\left(z\right) - 
    Q_{\nu+1}\left(z\right)\right] 
  &= -\frac{2}{\pi}\sin(\nu\pi)\left[
    \psi(\nu+2)-\psi(\nu+1) \right] \nonumber \\
  &= -\frac{2}{\pi (\nu+1)}\sin(\nu \pi).
  \label{smallermess2}
\end{align}
Putting together~\eqref{smallermess1},~\eqref{smallermess2},
and~\eqref{curlT}, we have
\begin{align} 
  \label{doublelayerwithz}
\lim_{\theta \rightarrow 0^+}F(\theta) 
  = \lim_{z\rightarrow 1-}\C(\nu+1)\left[ P_{\nu+1}(-z)+zP{\nu}(-z)\right]
  = \C \frac{2}{\pi} \sin(\nu \pi).
\end{align}
\end{proof}
\qed

Observe that setting $\C=-\frac{1}{4\sin(\nu \pi)}$, that is, using
exactly the constant defined in~\eqref{constant-definition} in the
definition of $G_k(x,x_0)$, we have $\lim_{\theta\rightarrow
0^+}F(\theta)=\frac{1}{2\pi}.$

\subsection{Representation formula}
\label{s:representation}
Let $\Omega$ and $\Omega^{c}:=\S\setminus{\overline{\Omega}}$ be a
simply-connected submanifold and its complement on the sphere. We can
then derive a representation formula in terms of the fundamental
solution
$G_k(x,x_{0})$ of the previous section.
\begin{proposition}
\label{prop:repr}
Every sufficiently smooth function $u \in C^{2}(\Omega) \cap
C^{1}(\overline{\Omega})$ satisfies the representation formula
\begin{align} 
\label{eq:representationformula}
 - \int\limits_{\Omega} 
    G_k(x,x_{0})\left[\lap_{\S}u(x) -k^2 u(x)\right]\,d\sigma_x 
 &- \int\limits_{{\Gamma}} G_k(x,x_{0})  
    \underline{\mbox{curl}}_{\S} u(x) \cdot \vec{t}(x) ds_x \nonumber \\
 + \int\limits_{{\Gamma}}  u(x)
 \,\underline{\mbox{curl}}_{{\S}} G_k(x,x_0) \cdot \vec{t}(x) ds_x 
 &= \left\{ \begin{array}{ll} u(x_0) & \mbox{ if $x_0\in \Omega$},\\
  0 & \mbox{ if  $x_0\in{\Omega}_c$}. 
  \end{array} \right.
\end{align}

\end{proposition}

\begin{proof} 
We let $x\in \Omega$, and $x_0$ be a point either in $\Omega$ or
$\Omega^c$.  We proceed by examining the two possible cases.

{\bf Case 1}:  $x_{0} \in \Omega^{c}$.  In this situation, the
fundamental solution $G_{k}(x,x_{0})$ satisfies the Yukawa-Beltrami
equation point-wise at all $x\in \Omega$. Since $x \neq x_0$,
$G_{k}(x,x_0)$ is a smooth function and we can therefore use Green's
second identity~\eqref{Green2} with $v(x)$ replaced by
$G_{k}(x,x_{0})$. Since $x_0\not \in \Omega$ all the
integrals involved are bounded, and 
\begin{align*}
  \int\limits_{\Omega} u(x)(\lap_{\S}G_{k}(x,x_0) 
      -k^{2}G_{k}(x,x_0)) \, d\sigma_x=0.
\end{align*}
Therefore, using~\eqref{Green2} we have 
\begin{align*}
 0 &= - \int\limits_{\Omega}
    G_k(x,x_{0})\left(\lap_{\S}u(x) -k^2 u(x)\right)\,d\sigma_x \\
    &-\int\limits_{{\Gamma}} G_k(x,x_{0}) \, \underline{\mbox{curl}}_{\S}
    u(x) \cdot \vec{t}(x) ds_{x} 
 + \int\limits_{{\Gamma}}  u(x)
 \,\underline{\mbox{curl}}_{{\S}} G_k(x,x_0) \cdot \vec{t}(x) ds_{x}.
\end{align*}
\begin{figure}
  \centering
  \includegraphics[width=0.6\textwidth]{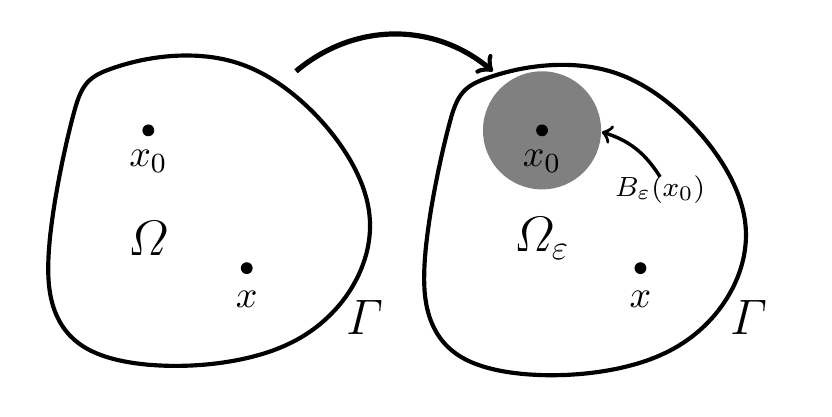}
  \caption{\label{f:proof1} The points $x,x_0$ belong to $\Omega$. We
  denote by $\Omega_\varepsilon:=\Omega\setminus B_{\varepsilon}(x_0)$}
\end{figure}

{\bf Case 2}:  $x_0\in \Omega$.  In this situation  we cannot directly
use Green's second identity, since the fundamental solution has a
logarithmic singularity when $x=x_0$. Instead, we first define the
$\varepsilon$-neighbourhood of $x_0$,
$B_{\varepsilon}(x_0):=\{y\in{\S}:\, |y-x_0| \,<\, \varepsilon\}$. We
choose $\varepsilon>0$ small enough such that
$B_\varepsilon(x_{0})\subset \Omega$ (see Figure~\ref{f:proof1}). We
then set ${\Omega}_{\varepsilon}:=\Omega\setminus B_{\varepsilon}(x_0)$.
For $x\in \Omega_\varepsilon$, once again $G_k(x,x_0)$ satisfies the
Yukawa-Beltrami equation exactly.   The second Green's formula for
$\Omega_{\varepsilon}$ with $v(x)=G_k(x,x_0)$ yields:
\begin{align}
  - \int\limits_{\Omega_{\varepsilon}} G_k(x,x_0)\,
    (\lap_{\S}& u(x) -k^2u(x))\,d\sigma_x = 
    \int\limits_\Gamma [ G_k(x,x_0)\,\underline{\mbox{curl}}_{\S} 
    u(x) - u(x)\,\underline{\mbox{curl}}_{\S} G_k(x,x_0)] \cdot 
    \vec{t}(x) ds_x \nonumber \\
  &+\int\limits_{\partial B_{\varepsilon}(x_0)} [
    G_k(x,x_0)\,\underline{\mbox{curl}}_{\S} u(x) -
    u(x)\,\underline{\mbox{curl}}_{\S} G_k(x,x_0)]\cdot 
    \vec{t}(x) ds_{x}.
  \label{Green_S1eps}
\end{align}
We shall discuss each of these terms. We first note that since $x_{0}\in
\Omega$, the integrals over $\Gamma$ in~\eqref{Green_S1eps} are
well-defined. So we must examine the volume integral over
$\Omega_\varepsilon$ and the integrals over $\partial B_\varepsilon$.
For ease of exposition and without loss of generality, we take
$x_{0}=(0,0,1)$.  The curve $\partial B_{\varepsilon}$ is then fully
described by the latitude $\theta_{\varepsilon}$.  In this case, if
$x\in \partial B_\varepsilon$, then $<x,x_{0}> =
\cos(\theta_\varepsilon)$, and  $G_k(x,x_0) = \C
P_{\nu(k)}\left(-<x,x_{0}>\right) = \C
P_{\nu(k)}\left(-\cos(\theta_\varepsilon)\right).$ Therefore, since
$\cos(\theta_\varepsilon) = 1- \theta^2_\varepsilon/2 + \cdots$, we can
use the asymptotic result~\eqref{e:logSingularity} to compute
\begin{align*}
 \int\limits_{\partial B_{\varepsilon}}\vert G_k(x,x_0)\,\vert  ds_x 
  &\leq \int_{0}^{2\pi}\vert \C P_{\nu(k)}(-\cos\theta_{\varepsilon})  
    \, \sin\theta_{\varepsilon}| \, d\varphi \\
  &= \C|\sin\theta_{\varepsilon}|
    |P_{\nu(k)}(-\cos\theta_{\varepsilon}) | 2\pi \\
  &= 2 \pi |\sin\theta_{\varepsilon}|\C \left(\cos(\nu\pi) 
    P_{\nu(k)}(\cos \theta_\varepsilon) - 
    \frac{2}{\pi}\sin(\nu \pi)
    Q_\nu(\cos \theta_{\varepsilon})\right) \\
  &\rightarrow 0 \quad \mbox{ (as $\varepsilon\rightarrow 0$)},
\end{align*}
where the limit holds since the singularity of $P_{\nu(k)}$ is only
logarithmic.  We can therefore estimate the first integral over
$\partial B_\varepsilon$ in~\eqref{Green_S1eps} as
\begin{align*}
  \left|\int\limits_{\partial B_{\varepsilon}} 
  G_{k}(x,x_{0})\,\underline{\mbox{curl}}_{\S} 
  u(x)\cdot\vec{t}(x)\, ds_x\right|\, \leq\, 
  \|\underline{\mbox{curl}}_{\S}u\|_{L^{\infty}} \,
  \int\limits_{\partial B_{\varepsilon}} |G_k(x,x_0)|\,ds_x 
  \rightarrow 0, \, \mbox{as}\, \varepsilon \rightarrow 0.
\end{align*}

To analyse the second contribution along $\partial B_{\varepsilon}$
in~\eqref{Green_S1eps}, we again assume $x_{0}=(0,0,1)$.
Using~\eqref{curlT} we deduce that (note the orientation of $\partial
B_{\varepsilon}$)
\begin{align}
  \label{secondterm}
  -\int\limits_{\partial B_{\varepsilon}}
    [u(x)&\underline{\mbox{curl}}_{\S}G(x,x_0)\cdot \vec{t}(x)] ds_x 
    \nonumber \\
  &=\frac{\C(\nu+1) \sin^{2}\theta_{\varepsilon}}
         {\cos^{2}\theta_{\varepsilon} -1} 
  \left[ P_{\nu+1}(-\cos\theta_{\varepsilon}) +   
  (\cos \theta_\varepsilon) P_{\nu}(-\cos \theta_\varepsilon)\right] 
  \int_{2\pi}^0 u(\varphi,\theta_{\varepsilon})\, d\varphi  \nonumber\\
  &= \C(\nu+1) \left[ P_{\nu+1}(-\cos\theta_\varepsilon)
  + (\cos \theta_\varepsilon) P_\nu(-\cos\theta_\varepsilon)\right] 
  \int_{0}^{2\pi} u(\varphi,\theta_{\varepsilon})\, d\varphi.
\end{align}

Combining~\eqref{smallermess1},~\eqref{smallermess2}
and~\eqref{secondterm} and using the continuity of $u$, we have
\begin{align*}
  -\lim_{\varepsilon\rightarrow 0}\int\limits_{\partial B_{\varepsilon}}
  u(x)& [\underline{\mbox{curl}}_{\S}G(x,x_0)\cdot \vec{t}(x)] ds_{x}
  = -4\C\sin(\nu \pi)  u(x_0) = u(x_0).
\end{align*} 
Note that $\C=-\frac{1}{4\sin(\nu\pi)}$ is chosen precisely to be this
value in~\eqref{constant-definition} to ensure this integral, in the
limit, is $u(x_0)$.

We finally observe that since $u$ is assumed to be smooth enough,
\begin{align*} 
  \left| \int\limits_{\Omega} G_k(x,x_0)\,(\lap_{\S} u(x)-k^2u(x)) 
  \, d\sigma_x\right|\,&\leq \,
  \|\lap_{\S}\,u -k^2 u\|_{L^{\infty}(\Omega)} 
  \int\limits_{\Omega} |G_k(x,x_0)| \, d\sigma_x\,\\
  & \leq \|\lap_{\S}\,u -k^2 u\|_{L^{\infty}(\Omega)}M,
\end{align*}
where $M$ is a constant which depends on the volume of $\Omega$.
Specifically, since $P_\nu(k)(-t)$ is smooth away from the
(logarithmic) singularity at $t=1$, we have 
\begin{align*}
  M := \int\limits_{\Omega} |G_k(x,x_0)| \, d\sigma_x
  = \int_{B_{\varepsilon}(x_0)}|G_k(x,x_0)| \, d\sigma_x + M_1.
\end{align*}
The first integral is bounded since the logarithm is integrable over
${B_{\varepsilon}(x_0)}$, and $M_1$ is some finite constant depending
on the area of $\Omega$. Hence $M$ is a finite positive
constant depending on $\Omega$ and
\begin{align*}
  \lim_{\varepsilon\rightarrow 0}
  \int\limits_{\Omega_{\varepsilon}} G_k(x,x_0)\,(\lap_{\S} u(x)
  -k^2u(x)) \,d\sigma_x = \int\limits_{\Omega} G_k(x,x_0)\,
  (\lap_{\S} u(x) -k^2u(x))\,d\sigma_x.
\end{align*}

Therefore, taking limits in~\eqref{Green_S1eps} proves the result.
\end{proof}
\qed

\section{Layer potentials and boundary integral operators}
Now that we have a convenient parametrix $G_k(x,x_0)$ defined in
Definition~\ref{fundamentaldef} for the Yukawa-Beltrami equation and a
representation formula~\eqref{eq:representationformula}, we can define
convenient layer potentials which in turn will be used to reduce the
boundary value problem~\eqref{DBVP}  over the domain $\Omega$ to a
boundary integral equation over $\Gamma = \partial \Omega$.

\subsection{Single- and double-layer potentials}
We define the following two layer potentials.
\begin{itemize}
\item The {\bf single-layer potential} with sufficiently smooth density
function $\sigma$:
  \begin{align}
    (\widetilde{V_k}\sigma)(x):=\int\limits_{\Gamma}
    G_k(x,y)\,\sigma(y)\,ds_y,\quad\mbox{for $x \notin \Gamma$};
    \label{singlelayerpotential}
  \end{align}
\item and the {\bf double-layer potential} with sufficiently smooth
density function $\mu$:
  \begin{align}
    (\widetilde{W_k}\mu )(x):=\int\limits_{\Gamma}
    \mu(y)\,[\underline{\mbox{curl}}_{\S} G_k(x,y)\cdot \vec{t}(y)]\,
    ds_y, \quad\mbox{for $x\notin\Gamma$},
    \label{doublelayerpotential}
  \end{align}
\end{itemize}
We wish to point out here that the form of the double-layer potential above is equivalent to the, perhaps, more familiar form:
  \begin{align}
    (\widetilde{W_k}\mu )(x):=-\int\limits_{\Gamma}
    \mu(y) \frac{\partial \,}{\partial n_y} G_k(x,y) \,
    ds_y, \quad\mbox{for $x\notin\Gamma$}.
  \end{align}
where $n_{y}$ is the outward pointing normal to $\Gamma$ at the point
$y$ lying in a tangent plane to $\S$.  See~\cite{kro:nig2013} for a
more detailed discussion on this point.

By Proposition~\ref{prop:repr}, every solution to the homogeneous
Yukawa-Beltrami equation can be written as the sum of a single- and a
double-layer potential.  This  is the starting point for the so-called
direct boundary integral approach. However, for the purpose of this
paper we follow the layer ansatz based on the following observation.

For $x\notin\Gamma, x\in \Omega$, the single-layer potential
in~\eqref{singlelayerpotential} satisfies:
\begin{align*}
  (-\lap_{\S} + k^{2}) (\widetilde{V_k}\sigma)(x)& =
  (-\lap_{\S} + k^2) \int\limits_\Gamma G_k(x,y)\, \sigma(y)\,ds_{y} \\
  & = \int\limits_\Gamma (-\lap_{\S}+k^2) G_k(x,y)\, 
    \sigma(y)\,ds_{y}= \, 0.
\end{align*}

Hence, we may find the general solution of the Dirichlet boundary value
problem~\eqref{DBVP} in terms of a single-layer potential
\begin{align*}
  u(x) & =  \int\limits_\Gamma G_k(x,y) \,\sigma(y) \,ds_{y}. 
\end{align*}
We would then need to calculate the unknown density $\sigma$.
 
Similarly, the double-layer potential in~\eqref{doublelayerpotential}
satisfies the Yukawa-Beltrami equation for $x\in \Omega, x\notin\Gamma$
\begin{align*}
  (-\lap_{\S} + k^2)(\widetilde{W}\mu)(x) & = (-\lap_{\S}+k^2)
  \int\limits_\Gamma \mu(y)\, [\underline{\mbox{curl}}_{\S}
  G(x,y)\cdot\vec{t}(y)]\,ds_{y} \\
  &=\int\limits_\Gamma \mu(y)
  [(\lap_{\S}-k^2)\,\underline{\mbox{curl}}_{\S}\,
  G(x,y)\cdot\vec{t}(y)]\,ds_{y} \\ 
  &=\int\limits_\Gamma \mu(y) [\underline{\mbox{curl}}_{\S} \,
  (\lap_{\S} -k^2)G(x,y)\cdot\vec{t}(y)]\,ds_{y} = 0.
\end{align*}
We might thus also try to look for the solution to~\eqref{DBVP} in the
form of a double-layer.

\subsection{Jump relations for the layer potentials} 
In the previous section, we have only defined the layer potentials for
$x$ away from the boundary curve. However, in order to align the
operators with the given Dirichlet data along $\Gamma$, we need to
investigate their behavior in the limit as $x$ approaches $\Gamma$.
Similarly, if one is interested in solving the Neumann problem in which
the tangential component of the vectorial surface rotation is prescribed
along $\Gamma$, one has to investigate the limit features of this
quantity for the layer potentials. In both cases, there will be certain
jump relations across the curve $\Gamma$. For the purpose of this paper
however, we will restrict ourselves to the Dirichlet case.  First,
consider the single-layer potential with density $\sigma$ for $x \notin
\Gamma$:
\begin{align*}
  (\widetilde{V_k}\sigma)(x) & =  \int\limits_\Gamma
  G_k(x,y)\,\sigma(y)\,ds_{y} \\
  & = \C \int\limits_\Gamma P_{\nu(k)}(-<x,y>)\,\sigma(y)\,ds_{y}.
\end{align*}
The following Lemma describes the limit behavior of the single-layer
potential.  \begin{lemma} Let $\widetilde{V_k}$ be the single-layer
potential defined in~\eqref{singlelayerpotential}. For $x_{0}\in\Gamma$
we have:
\begin{align*} 
  (V\sigma)(x_0) & := \lim\limits_{\substack{
      x \to x_{0} \\ x \in \Omega}}
  \,(\widetilde{V_k}\sigma)(x) = \int\limits_\Gamma G_k(x_{0},y)\,\sigma(y)\,ds_{y}
\end{align*}
as a weakly singular line integral and hence $(\widetilde{V_k}\sigma)$
is continuous across $\Gamma$.
\end{lemma}
\begin{proof}
Fix an arbitrary $\varepsilon > 0$. Let $x_0 \in \Gamma$ be fixed, and
$x \in \Omega$ satisfy $|x-x_0| < \varepsilon$.  Introduce the notation 
\begin{align*}
  C_{\e,\leq}:=\{y \in \Gamma,\|y-x_0\|\leq\e\},\qquad  
  C_{\e,>}:=\{y\in \Gamma,\|y-x_0\|>\e\}.
\end{align*}

Then, if we define
\begin{align*}
  I_{\varepsilon}(x) &: = \int\limits_\Gamma G_k(x,y)
  \,\sigma(y) \,ds_{y} - \int_{C_{\e,>}} G_k(x_0,y)
  \,\sigma(y)\,ds_{y},
\end{align*} 
we can easily show
\begin{align}
  I_{\varepsilon}& =  \int_{C_{\e,>}}
  \left[ G_k(x,y)\,-\, G_k(x_0,y) \right]\,\sigma(y)\,ds_{y} +    
    \int_{C_{\e,\leq}} G_k(x,y)\, \sigma(y)\, ds_{y}.
  \label{lemm}
\end{align}
The first integral in~\eqref{lemm} vanishes in the limit as
$x\rightarrow x_0$ since $P_{\nu(k)}(1-<x,x_0>)$ is continuous away from
$<x,x_0>=1$, i.e. where $x=x_0$.  That is, 
\begin{align*}
  \lim\limits_{x \to x_0} \int_{C_{\e,>}}
  \left[G_k(x,y)\,-\,G_k(x_0,y) \right]\,\sigma(y)\,ds_{y} = 0 .
\end{align*}
The second term in~\eqref{lemm} can be bound in terms of the density
$\sigma$:
\begin{align*}
  &\left|\int_{C_{\e,\leq}}
  G_k(x,y)\,\sigma(y)\,ds_{y}\,\right| \, \leq \,
  \| \sigma \|_{L_\infty(\Gamma)} \int_{C_{\e,\leq}}
  \left| G_k(x,y)\right|\,ds_{y}.
\end{align*}
To finish the proof, note that we can estimate
\begin{align*}
  \int_{C_{\e,\leq}} \left| G_k(x,y)\right|\,ds_{y}\,\leq\,
  \int\limits_{\substack{y \in \Gamma \\ |y-x| \leq 2\varepsilon}}
  \left| G_k(x,y)\right|\,ds_{y}
  \buildrel {x\rightarrow x_{0}}\over\longrightarrow 
  \int\limits_{\substack{y \in \Gamma \\ |y-x_{0}| \leq 2\varepsilon}}
  \left| G_k(x_0,y)\right|\,ds_{y}
  \buildrel{\varepsilon\rightarrow 0}\over\longrightarrow 0.
\end{align*}
As before, the final limit holds since $G_{k}(x_{0},y)$ has a
logarithmic singularity at $y = x_{0}$.  Putting these estimates
together, we see that  $\displaystyle \lim_{\varepsilon\rightarrow
0}\lim_{x\rightarrow x_{0}} I_{\varepsilon}(x)=0$, which proves the
assertion.
\end{proof}
\qed

The case of the double-layer potential is slightly more involved. We
anticipate that, just as for the Laplace-Beltrami double-layer
potential~\cite{gemmrich}, the Yukawa-Beltrami double-layer will possess
a jump across the boundary of a domain.  Indeed, since $G_k(x,x_0)$ has
a logarithmic singularity, the estimates follow the same argument as for
the Laplace-Beltrami. The details of the calculation are cumbersome, but
the overall strategy is that of Section 8.2 in~\cite{hackbusch}.

\begin{lemma} 
\label{l:DLPjump} 
Let $\widetilde{W_k}$ be the double-layer potential defined
in~\eqref{doublelayerpotential}, and let $\gamma_0^{\Omega}$ (resp.
$\gamma_0^{\Omega^c})$ denote the trace operator on $\Gamma$, with
traces from inside (respectively outside) $\Omega$.  For $x_{0} \in
\Gamma$ we have:
\begin{align*} 
  (\gamma^{\Omega^{c}}_0\widetilde{W_k}\mu)(x_0)=
  \lim\limits_{\substack{x \to x_0 \\ x \in \Omega^{c}}}\, 
  (\widetilde{W_k}\mu)(x) =
  (K\mu)(x_0)\,+\,\left(1-\frac{\alpha(x_0)}{2\pi}\right)\,\mu(x_0),
\end{align*}
where $\alpha(x_{0})$ represents the interior (with respect to
$\Omega^{c}$) angle of $\Gamma$ at $x_{0}$. For a smooth curve, $\alpha
= \pi$. The {\it double-layer operator} $K$ is defined via the Cauchy
principle value: 
\begin{align*}
  (K\mu)(x_0)&=\lim_{\varepsilon\rightarrow 0}
  \, (K_{\varepsilon}\mu)(x_0) \quad \mbox{where}\qquad
  (K_{\varepsilon}\mu)(x_0)&:=\int\limits_{y\in C_{\epsilon,\geq}} 
  \mu(y)\,\left[\underline{\mbox{curl}}_{\S}
  \,G_k(x_{0},y)\cdot\vec{t}(y)\right]\,ds_{y}.
\end{align*}
Hence the double-layer potential satisfies:
\begin{align*}
  \left[(\gamma_0\widetilde W_k\mu) \right]_\Gamma := 
  (\gamma^{\Omega^{c}}_0\widetilde W_k\mu)+(\gamma^{\Omega}_{0} 
  \widetilde{W_k}\mu)=\mu,
\end{align*}
where we tacitly assumed the orientation of the tangential vector
$\vec{t}$ along $\Gamma$ to be in accordance with the orientation of
$\Omega^{c}$ in the sense of Stoke's theorem.

\end{lemma}

\begin{proof}
We provide only a sketch of the proof.  Given $\varepsilon \, > \,0$,
let $x\in \Omega^{c}$ with $ \| x-x_0\|<\varepsilon$.   We use the same
notation for $C_{\epsilon,\geq}$ and $C_{\epsilon,<}$ introduced
earlier.  Then, for fixed $\e>0$, 
\begin{align}
  (\widetilde{W_k}\mu)(x)-(K_{\varepsilon} \mu)(x_0) = 
  \int_{C_{\e,\geq}}\mu(y)\left[\underline{\mbox{curl}}_{\S}\, 
  G_{k}(x,y)\,-\,\underline{\mbox{curl}}_{\S}\,G_k(x_{0},y)
  \right]\cdot\vec{t}(y)\, ds_{y} \nonumber\\
  +\int_{C_{\e,<}} \mu(y)\,\underline{\mbox{curl}}_{\S}\,
  G_k(x,y)\cdot\vec{t}(y)\,ds_{y}.
  \label{eq:Kjump1}
\end{align}
We note that the integrand of $\int_{C_{\e,\geq}} \mu(y)
\left[\underline{\mbox{curl}}_{\S}\,
G_k(x,y)\,-\,\underline{\mbox{curl}}_{\S}\,G_k(x_0,y)\right]\cdot\vec{t}(y)\,
ds_{y}$ is continuous in $x$ away from  $x_0$, and therefore 
\begin{align*}
  \lim_{x\rightarrow x_0}  \int_{C_{\e,\geq}}
  \mu(y)\left[\underline{\mbox{curl}}_{\S}\,
  G_{k}(x,y)\,-\,\underline{\mbox{curl}}_{\S}\,
  G_{k}(x_{0},y)\right]\cdot\vec{t}(y)\,ds_{y} = 0.
\end{align*}

The integral over $C_{\e,<}$ in~\eqref{eq:Kjump1} can be rewritten as
\begin{align}
  \int_{C_{\e,<}} \mu(y)\,\underline{\mbox{curl}}_{\S}\,
    G_k(x,y)\cdot\vec{t}(y)\,ds_{y} 
  = &\int_{C_{\e,<}} \left[\mu(y)-\mu(x_0)\right]\,
    \underline{\mbox{curl}}_{\S}\,G_k(x,y)\cdot\vec{t}(y)\,ds_{y}
    \nonumber \\
  + \mu(x_0)&\int_{C_{\e,<}} \underline{\mbox{curl}}_{\S}\,
    G_k(x,y)\cdot\vec{t}(y)\,ds_{y}.
  \label{interm}
\end{align}
For the first integral on the right hand side of~\eqref{interm}, using
Lemma~\ref{regularity}, we have the estimate (for a fixed $x$)
\begin{align*}
  &\left|\int_{C_{\e,<}} \left[\mu(y)-\mu(x_0)\right]\,
  \underline{\mbox{curl}}_{\S}\,G_k(x,y)\cdot\vec{t}(y)
  \,ds_{y}\right| \\
  &\qquad \leq \sup_{y \in C_{\e,<}} |\mu(y)-\mu(x_0)|
  \int_{C_{\e,<}} \left|\underline{\mbox{curl}}_{\S}\,
  G_k(x,y)\cdot\vec{t}(y)\right|\,ds_{y}\\
  &\qquad \leq M\,\cdot \mbox{length $(C_{\e,<})$}\;
  \cdot \sup_{y \in C_{\e,<}} |\mu(y)-\mu(x_0)|
\end{align*}
for some constant $M$, and hence the integral vanishes in the limit as
$\varepsilon\longrightarrow 0$.

For the second integral in~\eqref{interm}, we define
$\Omega_{\varepsilon}(x_{0}):=\left\{ y\in\Omega^{c}:\; \|x_0-y\|<
\varepsilon \right\}$ to see
\begin{align*}
  \mu(x_0)&\int_{C_{\e,<}} \underline{\mbox{curl}}_{\S}\,
    G_k(x,y)\cdot\vec{t}(y)\,ds_{y} \\
  &= \mu(x_0) \left[ \int\limits_{\partial \Omega_{\varepsilon}(x_0)} 
  \underline{\mbox{curl}}_{\S}\,G_k(x,y)\cdot\vec{t}(y)\,ds_{y} - 
  \int\limits_{\substack{y \in \Omega^{c} \\ \|y - x_{0}\| =
  \varepsilon}}
\underline{\mbox{curl}}_{\S}\,G_k(x,y)\cdot\vec{t}(y)\,ds_{y} \right].
\end{align*}
Using the representation formula~\eqref{eq:representationformula} with
$u=1$ we have
\begin{align*}
  \mu(x_0)&\int_{C_{\e,<}} \underline{\mbox{curl}}_{\S}\,
    G_k(x,y)\cdot\vec{t}(y)\,ds_{y} \\
    &=\;\mu(x_0) \left(1 - \int\limits_{\Omega_{\varepsilon}} 
    k^2 G_k(x,y) d\sigma_{y}\right) 
   - \mu(x_0) \int\limits_{\substack{y \in \Omega^{c} \\ \|y - x_{0}\| =
  \varepsilon}}
\underline{\mbox{curl}}_{\S}\,G_k(x,y)\cdot\vec{t}(y)\,ds_{y}.
\end{align*}
Without loss of generality we can set $x_0$ to be the north pole, and
compute the last integral above  to find, for all $x_{0}$: 
\begin{align*}
  \lim_{\varepsilon\rightarrow 0}
  \int\limits_{\substack{y \in \Gamma \\ \|y - x_{0}\| = \varepsilon}}
  \underline{\mbox{curl}}_{\S}\,G(x,y)\cdot\vec{t}(y)\,ds_{y}
  =\frac{\alpha(x_0)}{2\pi}.
\end{align*}
Putting the parts together we see that
\begin{align*}
  \lim_{\varepsilon\rightarrow 0} \lim_{x\rightarrow x_0}
  \left((\widetilde{W}\mu)(x)-(K_{\varepsilon}\mu)(x_0)\right)
  =\left(1-\frac{\alpha(x_0)}{2\pi}\right)\mu(x_0).
\end{align*}
\end{proof}
\qed

From Lemma~\ref{l:DLPjump}, we see that the kernel of the double-layer
operator, $\underline{\mbox{curl}}_{\S}\, G_{k}(x,y)\cdot\vec{t}(y)$,
is continuous at $x \in \Gamma$ as a function of $y$.  This allows us
to conclude that the integral operator $\widetilde{W_k}$ is a compact
operator from $L^{2}(\Gamma)$ to itself.  The compactness, in addition
to the jump $1-\alpha(x_0)/(2\pi)$ guarantees that the double-layer
potential representation will result in a second kind Fredholm integral
equation with a compact operator.

We need to record one further property of this kernel, which will be
used to understand the convergence properties of our quadrature rule in
Section~\ref{s:numerics}.

\begin{lemma}
\label{regularity}
Let $x_0,x$ be points on the sphere, connected by the smooth curve
$\Gamma$. Let $x_0$ be fixed, and let $\Gamma$ be parametrized by
$x=x(s), s\in[-A,A]$ such that $x(0)=x_{0}$. Then the kernel of the
double-layer operator is continuously differentiable in $s$, but the
second derivative is unbounded. More precisely, the function 
\begin{align*}
  f(s):=\underline{\mbox{curl}}_{\S}\,
    G_{k}(x_0,x(s))\cdot\vec{t}(x(s)), \quad s\in[-A,A], s \neq 0,
\end{align*}
has the following properties:
\begin{itemize}
\item 
\begin{align*} 
  \lim_{\substack {s \rightarrow 0}} f(s) = 
   \lim_{\substack {x \rightarrow x_{0} \\ x \in \Gamma}}
   [\underline{\mbox{curl}}_{\S}G_{k}(x_{0},x) \cdot \vec{t}(x)] =  
   -\frac{1}{4\pi} \kappa \mathbf{t}\cdot(\mathbf{n}_{p}
    \times \mathbf{x}_0)
\end{align*}

Here $\kappa$ is the principle curvature of $\Gamma$ at $x_{0}$,
$\mathbf{x}_0$ is the 3-dimensional vector associated with the point
$x_0$ (assuming the origin is located at the center of the sphere),
$\mathbf{n}_p$ is the principle normal of $\Gamma$ at $x_0$ and
$\mathbf{t}$ is a three-dimensional vector, identified with $\vec{t}$.
\item $\frac{d}{ds}f(s)$ can be extended to be well-defined and
continuous at $s=0$.
\item $\frac{d^2}{ds^2} f(s)$ is unbounded as $s\rightarrow 0$, and
therefore $f$ cannot be extended to be a $C^{2}$ function on $\Gamma$.
\end{itemize}
\end{lemma}

\begin{proof}

We provide a simple argument via l'H\^{o}pital's rule in the case that
$\Gamma$ is a simple smooth closed curve. For this argument, it is
easier  to work with points and vectors in $\mathbb{R}^3$.  Let $x_0$,
$x$ be points on the unit sphere. We identify them with the
3-dimensional vectors $\mathbf{x}_{0}$, $\mathbf{x}$ respectively.  Let
$\mathbf{t}$ and $\mathbf{n}$ be the unit tangent and normal vectors
at the point $\mathbf{x} \in \Gamma$ lying in the tangent plane of
$\mathcal{S}$.  We note for future reference that
$\mathbf{x}=\mathbf{e_r}$, $\mathbf{x} \cdot \mathbf{x} = 1$ and that
$\mathbf{x} = \mathbf{n} \times \mathbf{t}$. Since $\Gamma$ is
parametrized by arc length $s$, we also note the following identities:

\begin{align*}
  \mathbf{t} = \frac{d\mathbf{x}}{ds}, \qquad 
  \frac{d\mathbf{t}}{ds} = \kappa \, \mathbf{n}_{p}.
\end{align*}
The second identity is one of the Frenet formulae, where $\kappa$ is
the curvature of the curve $\Gamma$ at the point $\mathbf{x}$ and
$\mathbf{n}_{p}$ is the principal normal to the curve.  From these, it
is straightforward to show that 
\begin{align*}
  \frac{d \, }{ds} \left( \mathbf{t} \times \mathbf{x} \right)=
  \kappa \mathbf{n}_{p} \times \mathbf{x} . 
\end{align*}

We now examine the kernel of the double-layer potential. Calculating the
three dimensional gradient of the fundamental solution $G$ yields
\begin{align*}
  \nabla G_k(x_0,x) &= \C \nabla P_\nu\left(
    \frac{\|x_0-x\|^2}{2}-1\right) = 
  -\C P'_\nu\left(\frac{\|x_0-x\|^2}{2}-1\right)
    (\mathbf{x}_{0}-\mathbf{x}),
\end{align*}
which we can decompose into the surface gradient plus a derivative in
the radial direction. We can write this more concisely as
\begin{align*}
\nabla G_k(x_0,x)=-\C P'_{\nu}(z) (\mathbf{x}_0-\mathbf{x}), 
  \qquad \mbox{where}\qquad z=\frac{\|{ x_0}-{x}\|^2}{2}-1.
\end{align*} 
The kernel of the double-layer operator is $f(s): =
[\underline{\mbox{curl}}_{\S}G_{k}(x_{0},x) \cdot \vec{t}(x)]=
\frac{\partial}{\partial n} G_k({x}_{0},{x})$, which in turn can be
written as 
\begin{align*}
 f(s)=\frac{\partial \, }{\partial n} G_{k}(x_{0},x) 
 =\nabla G_{k} \cdot \mathbf{n} = \nabla G_k\cdot \left( \mathbf{t}
 \times \mathbf{x} \right) 
 =-\C (\mathbf{x}_0-\mathbf{x})P'_{\nu}(z)\cdot \left(
 \mathbf{t} \times \mathbf{x} \right).
\end{align*}
Recalling that $P'_{\nu}(z)=\frac{-(\nu+1)}{1-z^2}(P_{\nu+1}(z)-z
P_{\nu}(z))$~\eqref{difflegp},
\begin{align}
  f(s) &= \C \frac{(\nu+1)}{1-z^2}\left(
  P_{\nu+1}(z)-zP_{\nu}(z)\right) 
  (\mathbf{x}_0-\mathbf{x})\cdot \left(
  \mathbf{t} \times {\mathbf x} \right) \nonumber\\
  &= 2\C(\nu+1)\frac{P_{\nu+1}(z)-zP_{\nu}(z)}{1-z} \left[
  \frac{(\mathbf{x}_0-\mathbf{x})}{\|\mathbf{x}_{0}-\mathbf{x}\|^2}
  \cdot \left( \mathbf{t} \times \mathbf{x} \right)\right].
  \label{kern1}
\end{align} 

Now, the quantity
\begin{align*}
  2\C(\nu+1)\frac{P_{\nu+1}(z)-zP_{\nu}(z)}{1-z}
  =4\C(\nu+1)\frac{P_{\nu+1}(z)-zP_{\nu}(z)}{4-\|{x_0}-{x}\|^2},
\end{align*}
has a well-defined limit as $x \rightarrow
x_{0}$~\eqref{doublelayerwithz}. In fact,
following~\eqref{doublelayerwithz}
\begin{align*}
  \lim_{x\rightarrow x_0} 
  \C[\underline{\mbox{curl}}_{\S}G_{k}(x_{0},x) \cdot \vec{t}(x)] 
  = \C(\nu+1)\lim_{z\rightarrow 1^-}
  P_{\nu+1}(-z)+zP_{\nu}(-z)) = -\frac{1}{2\pi}.
\end{align*}

The remaining term in~\eqref{kern1} is $q(x,x_0) :=
\frac{(\mathbf{x}_{0}-\mathbf{x})}{\|\mathbf{x}_{0}-\mathbf{x}\|^2}
\cdot \left( \mathbf{t} \times \mathbf{x} \right).$  This term can be
shown to be continuous as $\mathbf{x} \rightarrow \mathbf{x}_{0}$ by
l'H\^{o}pital's rule.  A first application of l'H\^{o}pital's rule to
evaluate $q(x,x_0)$ at the point of singularity, $x = x_{0}$, 
\begin{align*}
  \lim_{x \rightarrow x_{0}}q(x,x_{0}) & = 
  \lim_{x \rightarrow x_{0}} 
  \frac{ -\mathbf{t} \cdot \left(\mathbf{t} \times \mathbf{x} \right)
  + (\mathbf{x}_{0}-\mathbf{x}) \cdot \left( \kappa \mathbf{n}_{p}
  \times \mathbf{x} \right)}
  {-2(\mathbf{x}_{0} - \mathbf{x}) \cdot \mathbf{t}} \\
  & = \lim_{\mathbf{x} \rightarrow \mathbf{x}_{0}} 
  \frac{(\mathbf{x}_{0}-\mathbf{x}) \cdot \left(
  \kappa\mathbf{n}_{p} \times \mathbf{x}\right)}
  {(\mathbf{x}_{0} - \mathbf{x}) \cdot \mathbf{t}}. 
\end{align*}
Proceeding with a second application of l'H\^{o}pital's rule yields
\begin{align*}
  \lim_{x \rightarrow x_{0}} q(x,x_0)  
  & = \lim_{x \rightarrow x_{0}} \frac{-\mathbf{t} \cdot 
  \left( \kappa \mathbf{n}_{p} \times \mathbf{x} \right) +
  \mathcal{O}(\mathbf{x}-\mathbf{x})}
  {-2\mathbf{t} \cdot \mathbf{t} + \mathcal{O}(\mathbf{x}-\mathbf{x})}
  = \frac{1}{2} \mathbf{t} \cdot \left( \kappa \mathbf{n}_{p} \times 
  \mathbf{x}_{0} \right).
\end{align*}
This shows that $f(s)$ has a well-defined limit as $s\rightarrow 0$,
that is, as $x \rightarrow x_0$. The kernel of the double-layer
operator can therefore be made continuous in the arc-length parameter.

If we now examine $\frac{d}{ds}f(s)$, we obtain
\begin{align*} 
  f'(s)&= -\C \frac{d}{ds}\left( P'_\nu(z)(\mathbf{x_0}-\mathbf{x}) \cdot 
    \mathbf{n} \right) \\
  &=-C_{k} \left\{P''_{\nu}(z) \frac{dz}{ds}
    (\mathbf{x}_{0}-\mathbf{x})  \cdot \mathbf{n} - P'_{\nu}(z)
    (-\mathbf{t} \cdot \mathbf{n}) + P'_\nu(z)((\mathbf{x}_{0} -
    \mathbf{x}) \cdot \frac{d}{ds}\mathbf{n})\right\}.
 \end{align*}
Since $P_\nu(z)$ solves Legendre's equation~\eqref{LegendrePequation}
and since $\mathbf{x},\mathbf{t}, \mathbf{n}$ form an orthogonal set,
we have
\begin{align*}
  \lim_{\substack {s \rightarrow 0}} f'(s) =-C_k
  \left\{\frac{2zP'_{\nu}(z)-\nu(\nu+1)P_{\nu}(z)}{1-z^2}
  \frac{dz}{ds}(\mathbf{x_0}-\mathbf{x})\cdot \mathbf{n} + \kappa
  P'_\nu(z)((\mathbf{x}_{0}-\mathbf{x})\cdot (\mathbf{n}_{p} \times 
  \mathbf{x})\right\}
\end{align*}
Applying L'H\^{o}pital's rule to each of the terms above, we see that
$f'(s)$ has a well-defined and bounded limit at $s \rightarrow 0$.
Therefore, the kernel of the double-layer operator is differentiable in
the arc length parameter, with bounded derivative as $s \rightarrow 0$. 

However,
\begin{align*} 
  f''(s) &= -\C \frac{d^2}{ds^2}\left( P'_\nu(z) 
  (\mathbf{x}_{0}-\mathbf{x}) \cdot \mathbf{n} \right)
\end{align*}
is not bounded as $s\rightarrow 0$. This can be shown using calculations
similar to those above, and we do not include them here.
\end{proof}
\qed

\section{Numerical Examples}
\label{s:numerics}
In this section, we apply standard numerical methods to solve the
Fredholm integral equation of the second kind
\begin{align}
  \label{e:dlpBIE}
  \frac{1}{2}\mu(x_{0}) + \int_{\Gamma}
    [\underline{\mbox{curl}}_{\S}G_{k}(x_{0},y) \cdot
    \vec{t}(y)] \mu(y)ds_{y} = g(x_{0}), \quad x_{0} \in \Gamma,
\end{align}
and to evaluate the double-layer potential
\begin{align}
  \label{e:dlpuEval}
  u(x) = \int_{\Gamma} [\underline{\mbox{curl}}_{\S}G_{k}(x,y) 
    \cdot \vec{t}(y)] \mu(y)ds_{y}, \quad x \in \Omega,
\end{align}
where
\begin{align*}
  G_{k}(x,y) = \C P_{\nu(k)} \left(
    \frac{\|x - y\|^{2}}{2} - 1\right), \quad 
  \C=-\frac{1}{4\sin(\nu\pi)}.
\end{align*}
We have assumed that the boundary $\Gamma$ of the geometry is a smooth
function so that $\alpha(x_{0}) = \pi$, where $\alpha$ is defined in
Lemma~\ref{l:DLPjump}.

In~\eqref{e:dlpuEval}, since $x \notin \Gamma$, the integrand is
periodic and smooth.  Therefore, for a fixed $x\in \Omega$, the
trapezoid rule has spectral accuracy.  However, since the error grows as
$x$ approaches $\Gamma$, our reported errors are only measured at points
$x$ sufficiently far from $\Gamma$.  We test two quadrature formulas for
solving~\eqref{e:dlpBIE}.  First, we test the trapezoid rule which we
expect will achieve third-order accuracy since the integrand is once
continuously differentiable (Lemma~\ref{regularity}).  Second, we test a
high-order hybrid Gauss-trapezoidal quadrature formula designed for
functions that contain logarithmic singularities~\cite{alpert}.

For all the examples, we discretize each connected component of the
boundary with $N$ unknowns and solve the resulting linear system with
unrestarted GMRES and a tolerance of $10^{-11}$.  The error of the
Alpert quadrature formula is $\mathcal{O}(h^{16}\log h)$, and we use
Fourier interpolation to assign values to the density function $\mu$ at
points that are intermediate to the regular grid.

We present four numerical examples which we now summarize.
\begin{itemize}
  \item{\em{The effect of the quadrature rule}:} For a two-ply connected
  domain, we report a convergence study for the two quadrature formulas.
  We also establish that the number of GMRES iterations is independent
  of the mesh size.

  \item{\em{The effect of $k$}:} For the same two-ply connected domain,
  we examine the effect of the parameter $k$ on the condition number of
  the linear system corresponding to~\eqref{e:dlpBIE}, and its effect on
  the number of GMRES iterations.

  \item{\em{The effect of the geometry's curvature}:} We consider a
  simply-connected domain and vary the aspect ratio of the major to
  minor axis of the domain's boundary.  We examine the effect of this
  parameter on the conditioning and the number of GMRES iterations.

  \item{\em{A complex domain}:} We demonstrate that our method is able
  to solve the Yukawa-Beltrami equation in complex domains by
  solving~\eqref{DBVP} in a 36-ply connected domain, with an acceptable
  number of GMRES iterations.
\end{itemize}

\subsection{The effect of the quadrature rule}
We consider the two-ply connect geometry illustrated in
Figure~\ref{f:twoply}.  An exact solution is formed by taking the
Dirichlet boundary condition corresponding to the sum of two fundamental
solutions centered inside the two islands.  In Table~\ref{t:example1},
we report the number of GMRES iterations (this was independent of the
quadrature formula).  We see that the number of GMRES iterations is
independent of the mesh size, the error of the trapezoid rule has
third-order accuracy, and the error of the Alpert quadrature formula
quickly decays to the GMRES tolerance.

\begin{figure}[htps]
  \includegraphics[width=0.5\textwidth]{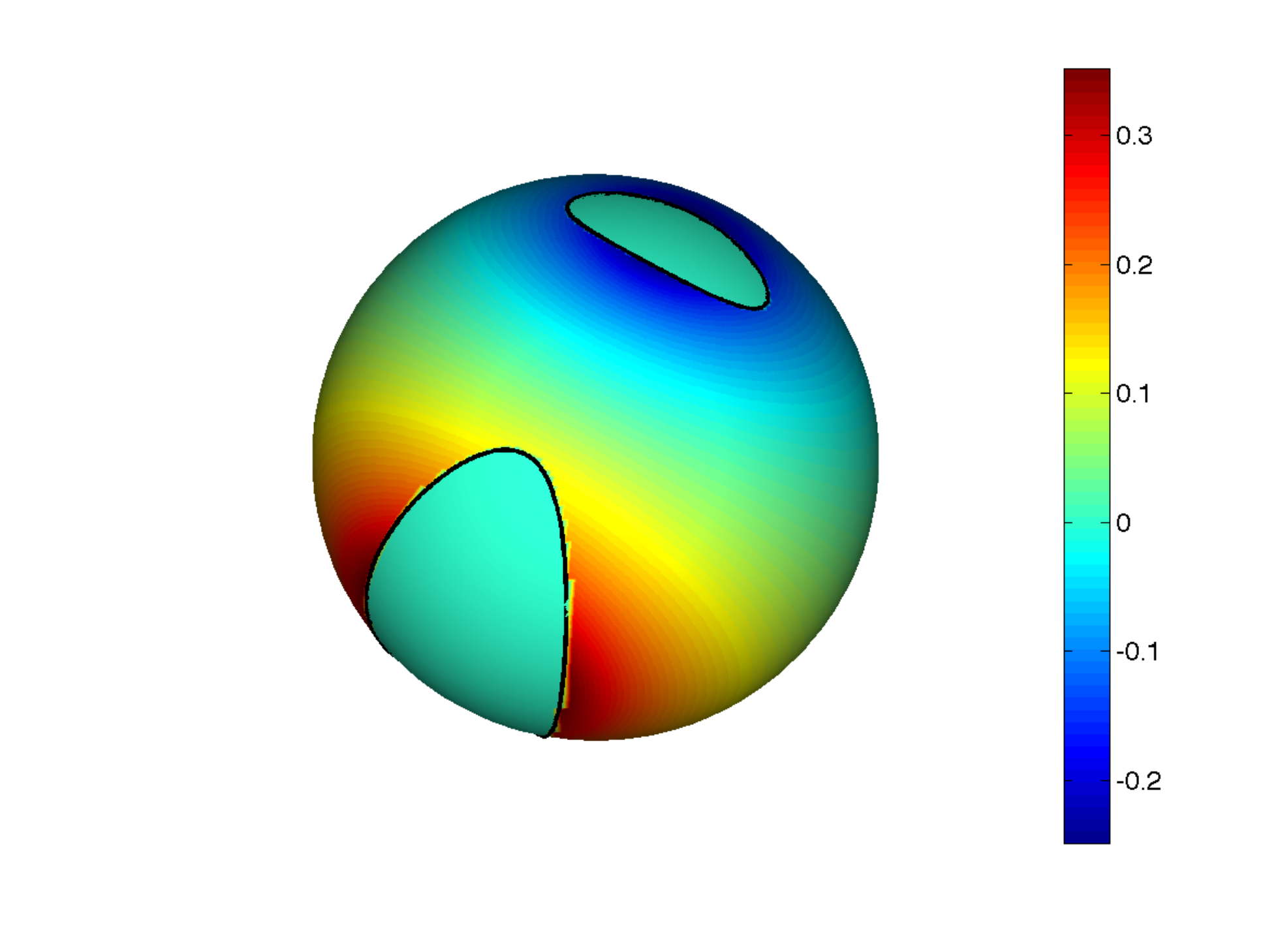}
  \includegraphics[width=0.5\textwidth]{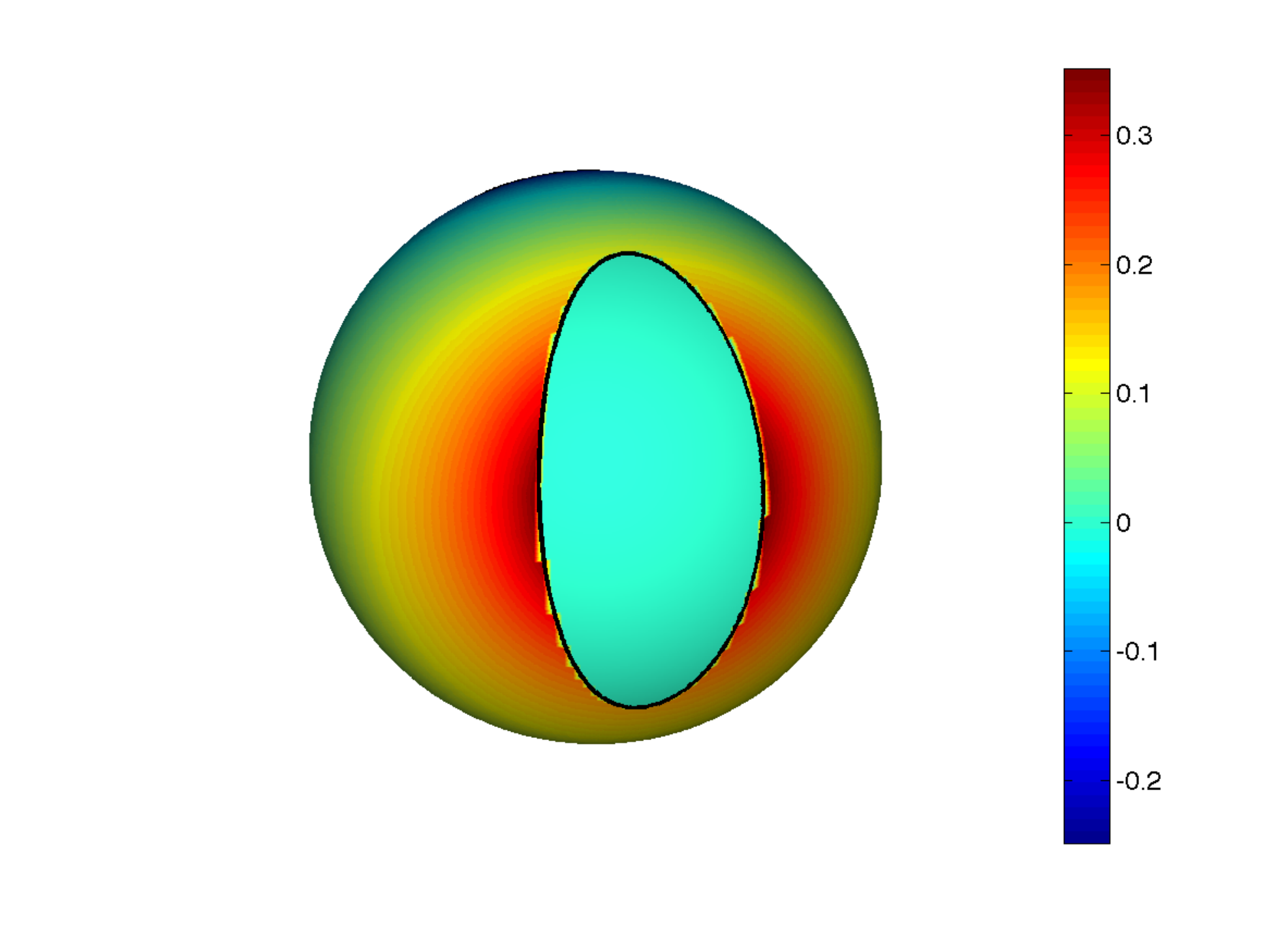}
\caption{\label{f:twoply} A numerical solution of the Yukawa-Beltrami
equation with $k=4$ in a two-ply connected domain viewed from two
different angles.  The exact solution is the sum of two fundamental
solutions, each one centered inside the regions bounded by $\Gamma$}
\end{figure}

\begin{table}[htps]
\caption{\label{t:example1} The number of GMRES iterations and the
error at a collection of points sufficiently far from $\Gamma$.}
\centering
\begin{tabular*}{0.8\textwidth}{@{\extracolsep{\fill}}lllll}
$N$ & \# GMRES & Trapezoid Error & Alpert Error  \\
\hline\noalign{\smallskip}
32   & 9 & 6.67E-5 & 5.54E-6  \\
64   & 9 & 7.84E-6 & 9.66E-10 \\
128  & 9 & 9.86E-7 & 2.28E-11 \\
256  & 9 & 1.24E-7 & 3.68E-11 \\
512  & 9 & 1.59E-8 & 1.76E-10 \\  
1024 & 9 & 1.84E-9 & 1.20E-10 \\ 
\noalign{\smallskip}\hline
\end{tabular*}
\end{table}

\subsection{The effect of $k$}
We consider the same two-ply connected geometry illustrated in
Figure~\ref{f:twoply}.  We solve~\eqref{e:dlpBIE} for varying values of
$k > 1/2$ using Alpert's quadrature rule with $N=32$.  We report the
condition number of the resulting linear system and the required number
of GMRES iterations in Table~\ref{t:example2}.  We see that for larger
values of $k$, the conditioning of the linear system improves, and the
number of GMRES iterations decreases.  In Figure~\ref{f:evalues}, we
plot the eigenvalues of the linear system for $k=1$ and $k=64$.  We see
that for larger values of $k$, the eigenvalues cluster more strongly
around $1/2$ resulting in a smaller number of GMRES iterations and a
smaller condition number.

\begin{table}[htps]
\caption{\label{t:example2} The condition number of the linear system
corresponding to a discretization of~\eqref{e:dlpBIE} and the number of
GMRES iterations as a function of the parameter $k$.}
\centering
\begin{tabular*}{0.8\textwidth}{@{\extracolsep{\fill}}lll}
$k$ & Condition Number & \# GMRES \\
\hline\noalign{\smallskip}
0.51 & 4.22E1 & 14 \\
1    & 1.22E1 & 13 \\
2    & 4.50E0 & 11 \\
4    & 2.17E0 & 9  \\
8    & 1.45E0 & 8  \\  
16   & 1.24E0 & 7  \\ 
32   & 1.13E0 & 6  \\
64   & 1.07E0 & 6  \\
\noalign{\smallskip}\hline
\end{tabular*}
\end{table}

\begin{figure}[htps]
\centering
\begin{tabular}{cc}
\includegraphics[width=0.5\textwidth]{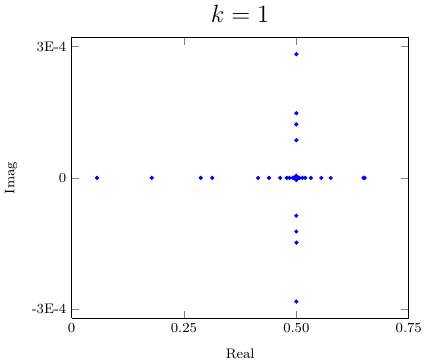} &
\includegraphics[width=0.5\textwidth]{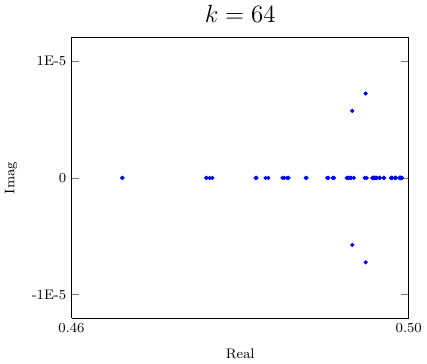}
\end{tabular}
\caption{\label{f:evalues} The eigenvalues of a discretization
of~\eqref{e:dlpBIE} for $k=1$ (left) and $k=64$ (right)}
\end{figure}

\subsection{The effect of the geometry's curvature}
We let $\Omega$ be exterior of an ellipse with a varying aspect ratio
of its major and minor axis.  The boundary $\Gamma$ is discretized with
$N=512$ points and is parameterized by $x(\alpha) = a\cos(\alpha)$,
$y(\alpha) = b\sin(\alpha)$, and $z(\alpha) =
\sqrt{1-x^{2}(\alpha)-y^{2}(\alpha)}$, where $a = 0.8$ and $b$ is
varied in Table~\ref{t:example3}.  We see that the curvature does have
an effect on the condition number of the corresponding linear system as
well as on the required number of GMRES iterations.

\begin{table}[htps]
\caption{\label{t:example3} The condition number of the linear system
corresponding to~\eqref{e:dlpBIE} and the number of required GMRES
iterations with respect to the ratio $a/b$.}
\centering
\begin{tabular*}{0.8\textwidth}{@{\extracolsep{\fill}}lll}
$\frac{a}{b}$ & Condition Number & \# GMRES \\
\hline\noalign{\smallskip}
1   & 1.11E0 & 2   \\ 
2   & 1.51E0 & 6   \\
4   & 2.88E0 & 8   \\
8   & 5.95E0 & 11  \\
16  & 1.23E1 & 19  \\
32  & 2.52E1 & 35  \\
64  & 4.72E1 & 54  \\
128 & 1.51E2 & 83  \\
256 & 1.43E3 & 227 \\
\noalign{\smallskip}\hline
\end{tabular*}
\end{table}

\subsection{A complex domain}
We take a 36-ply connected domain and set the boundary condition to be a
constant value of one everywhere.  We set the parameter of the PDE to be
$k=4$ and each boundary is discretized with $N=32$ points.  The
resulting linear system has 1152 unknowns, a condition number of $7.58$,
and GMRES requires 25 iterations to reach the desired tolerance of 11
digits.  A plot of the solution is in Figure~\ref{f:36ply}.

\begin{figure}
  \centering
  \includegraphics[width=0.5\textwidth]{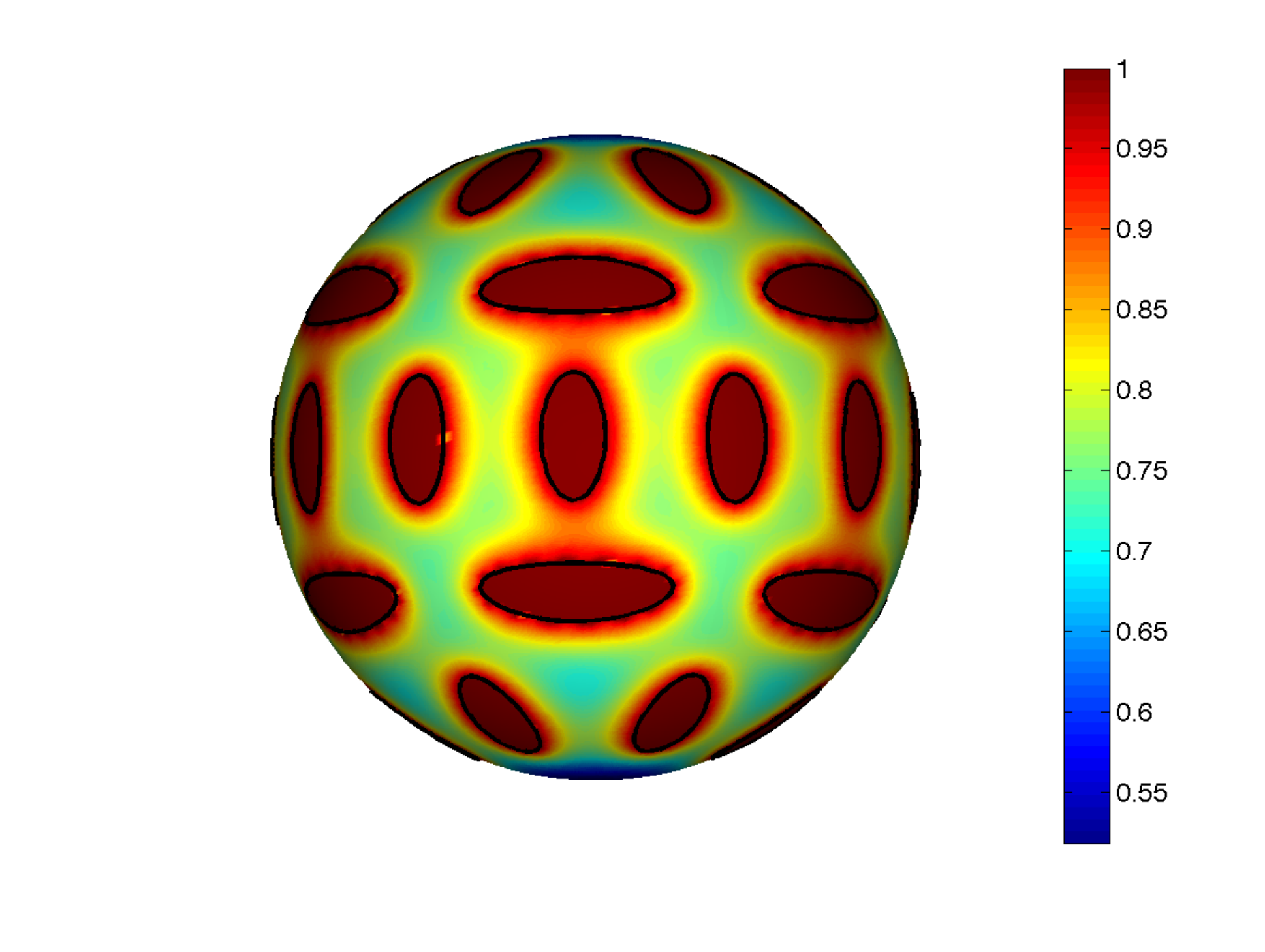}
  \caption{\label{f:36ply} The numerical solution of the Yukawa-Beltrami
  equation with $k=4$ in a 36-ply connected domain.  The Dirichlet
  boundary condition is equal to one on each of the boundaries (black
  curves)}
\end{figure}

\section{Conclusions and Discussion}
We have presented an integral equation strategy to solve the Dirichlet
boundary value problem for the Yukawa-Beltrami equation on a
multiply-connected, sub-manifold of the unit sphere.  The integral
equation formulation is based on a representation of a particularly
useful form of a parametrix for the Yukawa-Beltrami operator involving
conical functions.  Using a double-layer ansatz based on this
parametrix, a well-conditioned Fredholm equation of the second kind
arises.  Numerical experiments confirm the analytic properties of this
integral equation and by selecting appropriate quadrature rules, we are
able to compute highly accurate solutions.  This integral equation
formulation is amenable to acceleration either by a fast multipole
method or a fast direct solver; this is future work.  

The Yukawa-Beltrami equation arises when a temporal discretization is
applied to the heat equation.  However, the solution of~\eqref{bigmodel}
requires solving both a forced and a homogeneous problem.  While the
present work is designed to solve the homogeneous problem, future work
involves using volume potentials to form solutions to the forced
problem, as is done in~\cite{rothe:heat} for the heat equation in the
plane.

Other future work includes extending the presented methods to other
elliptic PDEs such as the Helmholtz or Stokes equations, and also to
other two-dimensional manifolds.  This will potentially create a new
class of methods for solving problems involving scattering or fluid
mechanics on the surface of smooth manifolds.

\begin{acknowledgements}
Supported in part by grants from the Natural Sciences and Engineering
Research Council of Canada. NN gratefully acknowledges support from the
Canada Research Chairs Council, Canada.
\end{acknowledgements}

\bibliographystyle{spmpsci}
\bibliography{refs}

\end{document}